\documentclass{amsart}

\usepackage{amsmath,amsthm,amssymb,enumerate,mdframed,lipsum,xypic,fancyhdr,comment}
\usepackage{graphicx}
\usepackage{circuitikz}
\usepackage{float}
\usepackage{todonotes}
\usepackage{hyperref}
\usepackage[capitalize,nameinlink]{cleveref}
\usepackage[foot]{amsaddr}

\Crefname{proposition}{Proposition}{Propositions}
\Crefname{lemma}{Lemma}{Lemmas}
\Crefname{theorem}{Theorem}{Theorems}
\Crefname{corollary}{Corollary}{Corollaries}
\Crefname{definition}{Definition}{Definitions}

\usepackage{appendix}

\usepackage[backend=bibtex]{biblatex}
\newcommand{\tinjsup}{t_{\inj}^{\sup}}
\newcommand{\tinj}{t_{\inj}}
\newcommand{\thom}{t_{\hom}}

\newcommand{\N}{\mathbb{N}}

\newcommand{\E}{\mathbb{E}\hspace{0.5mm}}

\newcommand{\GG}{\mathcal{G}}

\DeclareMathOperator{\Hom}{Hom}

\DeclareMathOperator{\inj}{inj}

\DeclareMathOperator{\ex}{ex}

\theoremstyle{plain}
\newtheorem{thm}{Theorem}[section]
\newtheorem{lem}[thm]{Lemma}
\newtheorem{prop}[thm]{Proposition}

\theoremstyle{definition}
\newtheorem{dfn}[thm]{Definition}
\newtheorem{rem}[thm]{Remark}
\newtheorem{claim}[thm]{Claim}
\newtheorem{cor}[thm]{Corollary}

    \title{Maximizing Subgraph Counts in Regular Graphs}
\author{Gabor Lippner \and Arturo Ortiz San Miguel}
\address{Department of Mathematics, Northeastern University}
\thanks{email: \texttt{g.lippner@northeastern.edu, ortizsanmiguel.a@northeastern.edu}}

\begin{document}

 \begin{abstract}
    Given a graph \(H\), we investigate the \(d\)-regular graphs \(G\) with the highest \(H\)-density via a novel spectral method. The key idea is to bound homomorphism numbers (which are typically \emph{non-spectral}) by spectral polynomials that are sharp for specific graphs. After relating injective homomorphism numbers from \(H\) to homomorphism numbers from quotient graphs of \(H\), this allows us to reframe the problem as a continuous optimization problem on the eigenvalues of \(G\).  
    
    For bipartite \(H\) and \(d\) large enough, we show \(G\) consists of disjoint copies of \(K_{d,d}\). For non-bipartite \(H\) and \(d\) sufficiently large, \(G\) is a collection of disjoint \(K_{d+1}\) graphs. For \(H=C_5\) and \(d=3\), disjoint Petersen graphs emerge.
   
    This positively answers a question from Cambie et al~\cite[Question 24]{cambie2022turan} regarding the generalized Tur\'an numbers $\ex(n,C_k,K_{1,d+1})$.
\end{abstract}

\maketitle

\bigskip
\noindent MSC Classification (2020): 
05C35 (Primary); 05C60, 05C50 (Secondary).  


\section{Introduction} \label{sec:intro}

Understanding the extremal behavior of subgraph counts in regular graphs, such as maximizing the number of triangles, cliques, or cycles is a fundamental question in extremal combinatorics. These problems are closely tied to the structure theory of graphs, pseudo-randomness, and spectral graph theory. They also serve as a lens through which we can better understand how local graph properties constrain global patterns.

Consider the case of triangles: in an \(n\)-vertex, \(d\)-regular graph, any node can be contained in at most \(\tbinom{d}{2}\) triangles. So the whole graph cannot have more than \(n/3 \cdot \tbinom{t}{2}\) and clearly equality is only achievable when each component of the graph is isomorphic to \(K_{d+1}\). A similar argument shows that the same is true for \(K_r\) subgraphs for any \(r \leq d\). On the other hand, for 4-cycles the answer is different. It is easy to see that an edge can be contained in at most \((d-1)^2\) 4-cycles, and this can only happen when the edge is in a component isomorphic to \(K_{d,d}\). So in order for a graph to have the largest number of \(4\)-cycles on average, it has to consist of disjoint copies of \(K_{d,d}\). 

In this paper we study the general question: given a fixed graph \(H\), which \(d\)-regular graphs have the ``most" copies of \(H\)? Of course, the subgraph count needs to be suitably normalized. Following~\cite{lovasz2012large}, we denote
\[ \inj(H,G) = \mbox{ the number of injective homomorphisms from \(H\to G\)}\] and refer to 
\[ \tinj(H,G) = \frac{\inj(H,G)}{|V(G)|}\] as the \(H\)-density of \(G\). This is the normalization that is best adapted to the case of bounded-degree graphs. 

Let \(\GG_d\) denote the set of all \(d\)-regular graphs. Our main goal can now be phrased as follows: for a fixed \(H\) and \(d > 0\), determine 
\[ \tinjsup (H) := \sup\{ \tinj(H,G) : G \in \GG_d\}\] and classify all graphs where the supremum is achieved. 

Based on the cases explained above, one might be led to believe that the answer is always the same: the optimal graphs must consist  of disjoint copies of \(K_{d+1}\) or of \(K_{d,d}\). The simplicity and intuitiveness of this answer perhaps explains why, until recently, such a basic question received so little attention in the literature.  However, it turns out that the situation is more complicated. Already for the case of the 5-cycle the answer depends on \(d\). For \(d \geq 7\) the above guess is indeed the right answer, but for \(d=3\) the optimal graphs consist of disjoint copies of the Petersen graph. And while for \(d=4,5,6\) we haven't been able determine the answer, we do know that copies of \(K_{d+1}\) are not optimal for $d=4,5$, and for $d=6$ there are other graphs that achieve the same density as $K_7$. 

\subsection{Related work}

The earliest reference we found where a somewhat similar question was studied is an unpublished note by Rivin~\cite{rivin_note}. There, a bound for the number of closed walks in a graph is given as a function of the number of its edges, and it is shown that the extremum is achieved when the edges form a clique.

A few works deal with the extremal behavior of more complicated parameters in regular graphs: Zhao~\cite{ZHAO_2010} studies the number of independent sets in $d$-regular graphs, and shows that this is maximized for unions of complete $d$-regular bipartite graphs, whereas Galvin~\cite{galvin_2015} shows that the number of $q$-colorings among $d$-regular graphs cannot asymptotically exceed that of a union of $K_{d,d}$s. 

Another connection to currently active research directions arises from casting the problem\footnote{The authors would like to thank Daniel Gerbner for pointing out this connection, and reference~\cite{cambie2022turan} in particular.} in the language of generalized Tur\'an numbers $\ex(n,T,H)$. Introduced by Alon and Shikhelman in~\cite{alon2016turan}, $\ex(n,T,H)$ denotes the maximum possible number of copies of $T$ in an $H$-free graph on $n$ vertices. It is a natural generalization of the classical Tur\'an number that considers $T=K_2$, that is, maximizing the number of edges in $H$-free graphs. 

This area has since received substantial attention. The recent survey~\cite{gerbner2026survey} contains an exhaustive list of known results for various choices of $T$ and $H$. In particular, Cambie, Verclos, and Kang~\cite{cambie2022turan} study the case of $H = K_{1,d+1}$, which simply restricts the graphs in question to maximum degree $\leq d$, and determine its value when $T$ is a tree.

\begin{thm}[\cite{cambie2022turan}, Proposition 22]\label{thm:tree_main}
Let \(H\) be a tree with maximum degree $\Delta(H)$ and diameter $q$. Then for any \(d \geq \Delta(H)\) the quantity $\ex(n,H,K_{1,1+d})$ is maximized by any $d$-regular graph with girth at least $q+1$. Equivalently, 
\[ \tinj(H,G) = \tinjsup(H) \Longleftrightarrow \mbox{ the girth of \(G\) is greater than the diameter of \(H\)}.\]
\end{thm}

Furthermore, \cite[Question 24]{cambie2022turan} asks whether, for a fixed $k$ and sufficiently large $d$, the quantity $\ex(n,C_k,K_{1,d+1})/n$ is maximized at disjoint unions of $K_{d,d}$ (when $k$ is even) and disjoint unions of $K_{d+1}$ (when $m$ is odd). 

Kirsch and Nir~\cite{kirsch2024turan} study the asymptotic behavior of $\ex(n,H,K_{1,d+1})$ for graphs $H$ that have a dominating vertex (which include cliques but not cycles). They show (in our terminology) that asymptotically $\ex(n,H,K_{1,d+1})/n$ is bounded by $\tinj(H,K_{d+1})$

We provide a positive answer to the question above, and a far reaching generalization and strengthening of the Kirsch-Nir result by showing that for arbitrary connected graphs $H$ that are not trees, for large enough $d$ the quantity $\ex(n,H,K_{1,d+1})/n$ has a unique connected maximizer: $K_{d,d}$ when $H$ is bipartite and $K_{d+1}$ when it is not. 

\subsection{Results}

Throughout this paper, unless stated otherwise, all graphs \(G = (V,E)\) are finite and simple. 
 all components of a graph \(G\) are isomorphic to a (connected) graph \(H\) then we will often refer to \(G\) as consisting of \emph{(disjoint) copies of \(H\)}.

Our main theorem confirms the intuition stated above when \(d\) is large enough. The answer is different depending on whether \(H\) is a tree (which was already known, see Theorem~\ref{thm:tree_main}) a non-tree bipartite graph, or a non-bipartite graph.

\begin{thm}\label{thm:bip_main}
    Let \(H\) be a connected bipartite graph that is not a tree. Then there is a \(d_0 = d_0(H)\) such that for all \(d \geq d_0\) and all \(G \in \GG_d\) we have 
    \[ \tinj(H,G) = \tinjsup(H) \Longleftrightarrow \mbox{\(G\) consists of copies of \(K_{d,d}\) }.\]
\end{thm}

\begin{thm}\label{thm:odd_cycles_main}
    Let \(H\) be a connected non-bipartite graph. Then there is a \(d_0 = d_0(H)\) such that for all \(d \geq d_0\) and all \(G \in \GG_d\) we have 
    \[ \tinj(H,G) = \tinjsup(H) \Longleftrightarrow \mbox{\(G\) consists of copies of \(K_{d+1}\) }.\]
\end{thm}

\begin{rem} It is worth noting that unless \(H\) is a tree, graphs that exactly maximize the \(H\)-density only exist for certain values of \(n\). Specifically, the optimal density can only be achieved for \(n=2kd\) for bipartite \(H\) and \(n= k(d+1)\) for non-bipartite \(H\). For large \(n\) that does not have the required number of vertices, we will see that the optimal graphs still have the same clique structure that covers all but a vanishingly small portion of their nodes. 
\end{rem}

From these results we can easily derive a positive answer to~\cite[Question 24]{cambie2022turan} in an even more general setting.

\begin{cor}\label{cor:q24}
    Let $H$ be a connected graph that is not a tree. Then, for $d$ sufficiently large, the unique connected maximizer of $\ex(n,H,K_{1,d+1})/n$ is $K_{d,d}$ when $H$ is bipartite, and $K_{d+1}$ when $H$ is not biparite. 
\end{cor}

\begin{proof}
    When $G$ is $d$-regular, the statement is trivially equivalent to Theorems~\ref{thm:bip_main} and~\ref{thm:odd_cycles_main}. Thus it remains to show the (intuitively obvious) fact that a non-regular $G$ cannot maximize $\tinj(H,G)$. Let $Y \subsetneq V(G)$ be the set of nodes that have the smallest degree in $G$.  Create a new graph $G'$ by taking two disjoint copies of $G$ and joining the two sets  corresponding to $Y$ in each copy by a matching. (When $G$ was bipartite, we can choose $G'$ to be bipartite by reversing the roles of the two parts in the copy, so the corresponding parts of $Y$ have the same size and can be matched.) Clearly, $G'$ is still connected and has maximum degree $\Delta(G') \leq d$. Furthermore, $\tinj(H,G) \leq \tinj(H,G')$ trivially. However, the minimum degree of $G'$, $\delta(G')$, is strictly bigger than $\delta(G)$. So after repeating this procedure $d-\delta(G)$ times, we get a connected $d$-regular graph $\hat{G}$ such that $\tinj(H,G) \leq \tinj(H, \hat{G})$. Since $\hat{G}$ cannot be either $K_{d+1}$ or $K_{d,d}$ (as those are the smallest $d$-regular, respectively bipartite $d$-regular, graphs) we see from the first case that $\tinj(H,G) \leq \tinj(H, \hat{G}) < \tinj(H,K_{d+1}/K_{d,d})$ as claimed.    
\end{proof}

Along the way we also study the (much simpler) problem of determining the maximum density of closed walks in \(d\)-regular graphs.  This is trivial for walks of even length (where the optimizers consist of copies of \(K_{d,d}\) for all \(d\)), but we can obtain a similarly general result for walks of odd length. We even show a version where only walks starting at a specific vertex are counted.
\begin{prop}\label{prop:vertexwise}
    For any \(d\geq 2\), the \(d\)-regular graph with the most closed \((2k+1)\)-walks containing a given vertex is \(K_{d+1}\).
\end{prop}

\begin{cor}
    For \(k \in \N\) and \(d \geq 2\): 
    \begin{enumerate}
        \item The \(d\)-regular graphs with the highest density of closed walks of length \(2k\) consist of copies of \(K_{d,d}\). 
        \item The \(d\)-regular graphs the highest density of closed walks of length \(2k+1\) consist of copies of \(K_{d+1}\). 
    \end{enumerate}        
\end{cor}

Finally, we investigate the case of \(H= C_5\) in a little more detail. Here, the bound obtained from \cref{thm:odd_cycles_main} is \(d_0 = 7\). The problem is then to find the optimal graphs for \(d=3,4,5,6\). We provide a solution for \(d=3\).
\begin{prop}\label{prop:Petersen}
    The unique connected 3-regular graph with the highest \(C_5\)-density is the Petersen graph.
\end{prop}

While for \(d=4,5,6\) we are not currently able to determine the optimal graphs, we do give examples that have more 5-cycles than \(K_{d+1}\) for these values of \(d\).

The rest of this paper is structured as follows: the rest of \cref{sec:intro} gives a proof overview and discusses some open directions. The main ingredients and the overall structure of the proof are explained in \cref{sec:spectral}.  The actual proof is broken into a part about homomorphism inequalities in \cref{sec:hom_ineq} and some rather technical results about polynomial optimization in \cref{sec:opt}. Finally, \cref{sec:appendix} focuses on the special case of \(H=C_5\).

\subsection{Methods}
The most straightforward way of showing that copies of \(K_{d,d}\) (respectively, \(K_{d+1}\)) are optimal would be to prove that the number of subgraphs containing a given vertex in a connected graph are maximized exactly when that graph is \(K_{d,d}\) (respectively, \(K_{d+1}\)). This works when \(H\) is a triangle, or more generally a complete graph and yields \(K_{d+1}\) as the optimum. 
The same `pointwise' maximization method works for \(H = C_4\) and yields copies of \(K_{d,d}\). However, it already fails for \(H=C_5\) and \(d=3\). As \cref{fig:figure 1} shows, it is possible to construct a 3-regular graph where a single edge is contained in five 5-cycles.  Nonetheless, the method can still be pushed through in this case, as \cref{fig:figure 1} shows, up to symmetries, the only way to have such an edge. Thus, if 
there is a 3-regular graph with an edge \(e\) containing the maximal, five, 5-cycles going through it then the neighborhood of this edge is isomorphic to that shown in \cref{fig:figure 1}. Thus, there is a cut edge with no 5-cycles going through it, and a simple computation shows that on average we get less than four 5-cycles per edge in this block. On the other hand the Petersen graph achieves four 5-cycles per edge, so it has to be optimal. We give a more detailed proof in \cref{sec:appendix}.

\begin{figure}[!ht]
\centering
\resizebox{.4\textwidth}{!}{%
\begin{tikzpicture}
  \node[circle, fill=black, inner sep=3pt, label=below:$u$] (0) at (8.75,8) {};
  \node[circle, fill=black, inner sep=3pt, label=below:$v$] (1) at (13.25,8) {};
  \node[circle, fill=black, inner sep=3pt, label=left:$z_1$] (2) at (8,10) {};
  \node[circle, fill=black, inner sep=3pt, label=below:$z_2$] (3) at (5.5,10.5) {};
  \node[circle, fill=black, inner sep=3pt,
  label=below:$y_1$] (4) at (11,11.25) {};
  \node[circle, fill=black, inner sep=3pt,
  label=below:$y_2$] (5) at (11,13) {};
  \node[circle, fill=black, inner sep=3pt] (6) at (11,14.75) {};
  \node[circle, fill=black, inner sep=3pt, label=right:$x_1$] (7) at (14,10) {};
  \node[circle, fill=black, inner sep=3pt, label=below:$x_2$] (8) at (16.75,10.5) {};
  \node[circle, fill=black, inner sep=3pt] (9) at (12.75,16) {};
  
  \draw (0) -- (1) node[pos=0.5,below] {\Large{\(e\)}};
  \draw (0) -- (2);
  \draw (3) -- (0);
  \draw (0) -- (1);
  \draw (1) -- (7);
  \draw (1) -- (8);
  \draw (2) -- (4);
  \draw (4) -- (7);
  \draw (3) -- (4);
  \draw (7) -- (5);
  \draw (8) -- (5);
  \draw (5) -- (2);
  \draw (3) -- (6);
  \draw (8) -- (6);
  \draw (6) -- (9);
\end{tikzpicture}
}%
\caption{A 3-regular graph with 5 edges through \(e\) and a cut edge.}
\label{fig:figure 1}
\end{figure}

This method becomes impractical (or even useless), even for \(H = C_5\) and \(d=4\). Instead, we reframe the problem as a continuous optimization problem on the eigenvalues of \(G\) via homomorphism numbers and injective homomorphism numbers. For example, as we will see in \cref{5-cycles} 
\[ \inj(C_5,G)= \sum_{i=1}^n \lambda_i^5 + (5-5d)\lambda_i^3,\] where \(\{\lambda_i\}\) denote the adjacency eigenvalues of \(G\). 
One can then try, for a fixed \(n\), to maximize the right hand side over all possible sets of \(\{\lambda_i\}\) that satisfy some basic identities due to them being the eigenvalues of a \(d\)-regular graph.  There is no apriori reason to expect that the optimum would be achieved at the spectrum of an actual graph, but it actually is, at least for \(d=3\) and \(d \geq 7\). 

However, for general \(H\), including \(H= C_k\) with \(k\geq 6\), the main challenge is that subgraph counts are not expressible using the spectrum of the graph. It is possible to express subgraph counts as linear combinations of homomorphism counts. Some of these are, indeed, expressible using the spectrum. Most notably homomorphisms of cycles. Unfortunately, for a general \(H\), most homomorphism counts that show up in these formulas are ``non-spectral''.  Our solution is to find suitable inequalities between homomorphism counts where the bounding expressions are spectral and that are sharp for copies of \(K_{d,d}\) (respectively \(K_{d+1}\)). This is challenging to do efficiently. In fact, in \cite{hatami2011undecidability}, it was shown that any linear inequality between homomorphism densities, which are defined using homomorphism numbers, can be shown using a (possibly infinite) number of Cauchy-Schwarz inequalities. However, deciding whether such an inequality is true is indeterminable. This highlights the difficulty of finding and proving useful homomorphism inequalities. We prove various infinite families of such inequalities that are sharp at the optimal graphs and can be expressed using the spectrum of \(G\). For a formal introduction to homomorphism numbers and their theory, see \cite{lovasz2012large}.

Finally, the spectral bounds we obtain for the subgraph counts this way can be optimized over all possible sets of spectra. It turns out that, for \(d\) large enough, the optimizer will again be the spectrum of copies of \(K_{d,d}\) (respectively \(K_{d+1}\)). This allows us to solve the problem for any \(H\), given that \(d\) is large enough.  Additionally, unlike the ``pointwise'' method, this spectral argument also show uniqueness of the optimum:  \cite{van2003graphs} shows that copies of \(K_{d+1}\), copies of \(K_{d,d}\), and copies of Petersen graphs are all determined by their spectrum. Since spectra of these graphs are the unique maximizer for these optimization problems, we have that these graphs are the unique optimizers.

\subsection{Further Directions}

Our work contributes not just specific results, but also a toolbox for attacking otherwise intractable extremal questions through a mix of spectral and combinatorial techniques that allow us to use analytical tools from continuous optimization. The general method is the following:
\begin{enumerate}
    \item Phrase the problem as one about a certain family of homomorphism numbers. 
    \item Find homomorphism number inequalities that are sharp at the claimed extremum such that every term is now spectral.
    \item Solve the constrained continuous optimization problem.
\end{enumerate}
For example, in an upcoming work, we find the \(d\)-regular graph with the most closed, non-backtracking walks of length \(k\) by using the fact that number of closed non-backtracking walks of length \(k\) in \(G\) can be determined from the number of locally injective homomorphisms from \(C_k\) to \(G\). This method can likely be used to approach types of Turan problems, as having no \(H\) subgraphs can be detected using the spectrum of a graph.

\subsubsection{Additional Constraints}
It would be a natural extension to find the extremum in certain natural subsets of all graphs, such as connected graphs, graphs with fixed girth, or graphs without certain forbidden subgraphs. We conjecture that among connected graphs the optimal one will almost have the same clique structure but with a small number of edges joining the cliques. In terms of the eigenvalue optimization, this constraint is equivalent to forcing exactly one eigenvalue equal to \(d\). Thus, we believe that for sufficiently large \(d\) and \(n\), the optimal connected graph is the graph whose spectrum is closest to the clique structure's spectrum. For odd \(k\), it has been shown that the \(d\)-regular graph of girth \(k\) with the most \(k\)-cycles is copies of the balanced blowup of \(C_k\) \cite{beke2024generalized}. We conjecture that the \(d\)-regular graph of girth \(g>3\) with the most \(k\)-cycles is copies of the balanced blowup of \(C_g\).

\subsubsection{Near Optimal Behavior}
In \cite{van2022regular}, it was shown that for \(H=K_3\), near optimal regular graphs are asymptotically almost surely disjoint cliques and one more connected component with $O(1)$ triangles. We believe the same holds for any non-bipartite~\(H\) when \(d\) is sufficiently large. Similarly, we conjecture the corresponding statement for bipartite \(H\) and \(K_{d,d}\). Even for $H=C_4$, this has shown to be challenging.

\section{Spectral optimization}\label{sec:spectral}

In this section we introduce the main tools and use them to derive \cref{thm:bip_main} and \cref{thm:odd_cycles_main}.

Given a graph \(G\), we denote its vertex set by \(V(G)\), its edge set by \(E(H)\), and its adjacency matrix by \(A_G\). The spectrum of \(A_G\) as a multi-set will be denoted by \(\sigma(G)\), and when necessary we will refer to the eigenvalues as \(\lambda_1, \dots, \lambda_n\) and the corresponding eigenvectors as \(\phi_i, \dots, \phi_n\), where \(n = |V(G)|\).

\subsection{Homomorphisms}

First we fix notation and recall the relevant definitions and results from~\cite{lovasz2012large} that will be used in the proofs. 

\begin{dfn} \mbox{ }
\begin{enumerate}
\item
Given two graphs \(H,G\), a homomorphism from $H$ to $G$ is a function \(f: V(H) \to V(G)\) that maps pairs of adjacent nodes to pairs of adjacent nodes. 
\item We denote \(\Hom(H,G)\) the set of homomorphisms from \(H\) to \(G\). Let $\hom(H,G) = |\Hom(H,G)|$, be the number of homomorphisms, and let \(\inj(H,G)\) be the number of injective homomorphisms. Furthermore we let 
\[ \thom(H,G) := \frac{\hom(H,G)}{|V(G)|} \ \mbox{ and }\ \tinj(H,G) := \frac{\inj(H,G)}{|V(G)|}\] and refer to these as (injective) homomorphism densities. 
\item Given a partition \(P\) of \(V(H)\), we denote \(H/P\) the graph obtained by identifying the nodes of \(H\) in each class of \(P\). In particular, if \(u,v\) were in the same class and \(uv\) was an edge, then \(H/P\) will have a loop edge. 
\end{enumerate}
\end{dfn}

In the language of homomorphism numbers, our task is to find  \(\sup_{G} \tinj(H, G)\) over all \(d\)-regular \(G\), and identify which graphs achieve the supremum.

In \cite{lovasz2012large}, Lovasz presents the following relations between these two numbers.

\begin{prop}[\cite{lovasz2012large}]\label{prop:hom to inj}

\begin{align*}
    \hom(H,G) = \sum_{P} \inj(H/P, G),
\end{align*}
where \(P\) ranges over all partitions of \(V(H)\).
\end{prop}

\begin{prop}[\cite{lovasz2012large}]\label{prop:inj to hom}

\begin{align*}
\inj(H,G) = \sum_P \mu_P \hom(H/P,G),
\end{align*}
where \(P\) ranges over all partitions of \(V(H)\) and
\begin{align*}
\mu_P = (-1)^{v(G)-|P|} \prod_{S \in P} (|S|-1)!,
\end{align*}
where \(|P|\) is the number of classes in \(P\) and \(S\) are the classes in \(P\).
\end{prop}

\subsection{Polynomial bounds}\label{sec:poly_bounds}

Our goal is to express \(\tinj(H,G)\) in terms of the spectrum of \(G\). \cref{prop:inj to hom} allows us to deal with homomorphisms instead of injective ones. Certain homomorphism counts can be expressed using the spectrum. For example,
\begin{equation*}\hom(C_k, G) = \mathrm{tr}(A_G^k) = \sum_{\lambda \in \sigma(G)} \lambda^k
\end{equation*}
is standard. 

The main challenge is that many quotients \(H/P\) cannot be expressed using the spectrum of \(G\). We call these \emph{non-spectral terms}. The idea is to bound these terms using the spectrum of \(G\) in a way that is sharp on the optimal graph. We are going to need upper bounds on terms with positive coefficients, and lower bounds for terms with negative coefficients in \cref{prop:inj to hom}. The bounds we obtain this way are summarized in the following statements.

\begin{thm}\label{thm:bound}
Let \(H\) be a connected graph that is not a tree. Then there is a 2-variable polynomial \(p_H(\lambda,d)\) such that the following are satisfied.
\begin{itemize}
\item For any \(d \geq |V(H)|\) integer and any \(G \in \GG(d)\):
\begin{equation}\label{eq:inj_bound}
    \inj(H,G) \leq \sum_{\lambda \in \sigma(G)} p_H(\lambda,d)
\end{equation}
\item \(p\) has degree \(|V(H)|\), and for some \(k > 0\) the polynomial \(p(\lambda,d) - \lambda^k d^{|V(H)|-k}\) has degree less than \(|V(H)|\). In other words, \(p\) has a single monomial with the highest degree and that monomial has coefficient 1, and it is not \(d^{|V(H)|}\). 
\item Depending on whether $H$ is bipartite or not:
\begin{enumerate}
\item When $H$ is bipartite, all monomials of \(p\) have an even \(\lambda\)-exponent, and if \(G\) consists of disjoint copies of \(K_{d,d}\) then \cref{eq:inj_bound} is satisfied with equality. 
\item When $H$ is not bipartite, the $k$ above is odd and  when \(G\) consists of disjoint copies of \(K_{d+1}\) then  \cref{eq:inj_bound} is satisfied with equality
\end{enumerate}
\end{itemize}
\end{thm}

The usefulness of such bounding polynomials becomes clear from the following lemmas. They imply that for a general class of polynomials, for \(d\) large enough, the sum \(\sum_{\lambda \in \sigma(G)}p(\lambda)\) is maximized among \(d\)-regular graphs exactly for copies of either \(K_{d+1}\) or \(K_{d,d}\). Instead of maximizing over the spectra of \(d\)-regular graphs, we can maximize over the larger set of random variables with values in \([-d,d]\) that have 0 mean and \(d\) second moment. For an \(n\)-vertex graph \(G\), let \(X_G\) denote the random variable whose distribution is the spectral measure of \(A_G\): \[\mathbb{P}(X_G = \lambda) = \frac{1}{n}\cdot(\mbox{the multiplicity of the eigenvalue \(\lambda\) in \(\sigma(G)\)}) .\]

\begin{lem}\label{lem:even_opt}
    Let \(p(x,y)\) be a degree \(n\) polynomial in which all monomials have an even \(x\)-exponent and such that some \(k=2m > 0\) the polynomial \(p(x,d) - x^k y^{n-k}\) has degree less than \(n\). Then there is a \(d_p\) that satisfies the following: If \(d\geq d_p\) and \(X\) is a random variable supported on \([-d,d]\) that satisfies \(\E X = 0\) and \(\E X^2 = d\), then
    \begin{equation}
        \E  p(X,d) \leq \E p(X_{K_{d,d}},d),
    \end{equation}
     and equality holds if and only if \(X \sim X_{K_{d,d}}\).
\end{lem}

\begin{lem}\label{lem:odd_opt}
    Let \(p(x,y)\) be a degree \(n\) polynomial such that some \(k = 2m+1 > 0\) the polynomial \(p(x,d) - x^k y^{n-k}\) has degree less than \(n\). Then there is a \(d_p\) that satisfies the following: If \(d\geq d_p\) and \(X\) is a random variable supported on \([-d,d]\) that satisfies \(\E X = 0\) and \(\E X^2 = d\), then
    \begin{equation}
        \E  p(X,d) \leq \E p(X_{K_{d+1}},d),
    \end{equation}
    and equality holds if and only if \(X \sim X_{K_{d+1}}\).
\end{lem}

With these tools in hand, \cref{thm:bip_main} and \cref{thm:odd_cycles_main} follow easily.

\begin{proof}[Proof of \cref{thm:bip_main} and \cref{thm:odd_cycles_main}]
Let \(H\) be fixed. By \cref{thm:bound}  there is a degree \(|V(H)|\) polynomial \(p_H(\lambda,d)\) with certain properties guaranteed by the theorem. Then, there is an integer \(d_{p_H}\) guaranteed by \cref{lem:even_opt} or \cref{lem:odd_opt}. Let \(d \geq \max\{d_{p_H},|V(H)|\}\), and let \(K\) denote \(K_{d,d}\) or \(K_{d+1}\) depending on whether we are in the bipartite or the non-bipartite case. Then, for any \(d\)-regular \(n\)-vertex \(G\) we get
\[ \tinj(H,G) \stackrel{(a)}{\leq} \frac{1}{n}\sum_{\lambda \in \sigma(G)} p_H(\lambda,d) \stackrel{(b)}{=} \E  p_H(X_G,d) \stackrel{(c)}{\leq} \E  p_H(K,d) \stackrel{(d)}{=} \tinj(H,K).\]
Here (a) is the statement of \cref{thm:bound}, (b) follows from \(X_G\) being the spectral measure of \(G\), (c) is the estimate from \cref{lem:even_opt}/\cref{lem:odd_opt}, and finally (d) holds because the polynomial \(p_H\) was constructed so that (a) is satisfied with equality for \(K\).

Lastly, to have equality, we need (c) to be sharp which only happens when \(X_G \sim X_K\) so the spectral measure of \(G\) is the same as that of \(K\). However, according to~\cite{van2003graphs} this only holds when \(G\) consist of copies of \(K\).
\end{proof}

\section{Homomorphism inequalities}\label{sec:hom_ineq}

In this section our goal is to prove \cref{thm:bound}. First establish various relationships between homomorphism counts that will be used later in the proof.

\begin{lem}\label{lem:hom_ineq}
Let \(H\) be a connected graph with \(n\) nodes and \(m\) edges. Let us fix a spanning tree \(T \leq H\). Let us write $E_0 = E(H) \setminus E(T)$ For each edge \(e \in E_0\) let \(H_e = T \cup \{e\} \leq H\) denote the subgraph with a single cycle that passes through \(e\). 
Let \(G \in \GG_d\). Then
    \begin{enumerate}
    \item\label{part:1} \(\inj(H,G) \leq \inj(H_e,G)\) for any \(e\in E_0\). There is equality for \(G = K_{d+1}\), and if \(H\) is bipartite then there is also equality for \(G = K_{d,d}\).
    \item\label{part:2} \(\hom(T,G) = d^{n-1}\cdot |V(G)|\) and
    \[ \hom(H_e,G) = d^{n-k} \cdot \sum_{\lambda \in \sigma(G)}\lambda^k ,\] where \(k\) is the length of the unique cycle in \(H_e\).
    \item\label{part:3} For any \(G \in \GG_d\)
    \[ \hom(H,G) \geq \left(\sum_{e \in E_0} \hom(H_e,G)\right) - (m-n)\hom(T,G).\] Equality holds for \(G = K_{d+1}\) and when \(H\) is bipartite then equality also holds for \(G = K_{d,d}\). 
    \item\label{part:4} For any \(G \in \GG_d\)
    \[ -\hom(H,G) \leq \sum_{\lambda \in \sigma(G)} \left( (n-m)d^{n-1} - \sum_{k\geq 3} c_k \lambda^k \right)\] where $c_k$ denotes the number of edges in $E_0$ that create a $k$-cycle. Again, equality holds for \(G = K_{d+1}\) and when \(H\) is bipartite then equality also holds for \(G = K_{d,d}\). 
    \end{enumerate}
\end{lem}

\begin{proof}
The inequality in \eqref{part:1} is trivial, since removing edges just makes it easier for functions to be homomorphisms. That there is equality for $G= K_{d+1}$ is because every injective function $f: V(H) \to V(K_{d+1})$ is, in fact, already a homomorphism, regardless of what edges are there in $H$. The same is true for bipartite $H$ and $G = K_{d,d}$ since we only have to consider functions that preserve the partitions.

To see \eqref{part:2}, note that if a node $v \in V(H_e)$ has degree 1, then $\hom(H_e,G) = d\cdot \hom(H_e - v, G)$ because any homomorphism $f \in \Hom(H_e-v,G)$ can be extended to a homomorphism $\tilde{f} \in \Hom(H_e,G)$ exactly $d$ different ways. Applying this repeatedly we get that $\hom(T,G) = d^{n-1}|V(G)|$ and $\hom(H_e,G) = d^{n-k}\hom(C_k,G)$ where $k\geq 3$ is the length of the unique cycle in $H_e$.  To finish the proof, use the simple identity 
\begin{equation*}\hom(C_k, G) = \mathrm{tr}(A_G^k) = \sum_{\lambda \in \sigma(G)} \lambda^k.
\end{equation*}

For part \eqref{part:3} we are going to use the following crude but obvious estimate. Let $x_1, \dots, x_s \in \{0,1\}$. Then 
\[ \prod_{i=1}^s x_i \geq 1-s+\sum_{i=1}^s x_i\] where equality holds when all $x_i$s are 1. 

For each edge $e \in E_0$ denote its two endpoints by $u_e$ and $v_e$. Note that a function $f : V(H) \to V(G)$ is a homomorphism from $H$ to $G$ if and only if it is a homomorphism from $T$ to $G$ and  
\[\prod_{e \in E_0} (A_G)_{f(u_e), f(v_e)}=1.\]
Hence 
    \begin{align*}
        \hom(H) &= \sum_{f 
        \in \Hom(T,G)}\prod_{e \in E_0} (A_G)_{f(u_e), f(v_e)} \\ &\geq \sum_{f \in \Hom(T,G)} \left( n-m +\sum_{e \in E_0} (A_G)_{f(u_e), f(v_e)}\right)\\
        &= \sum_{e\in E_0} \left(\sum_{f\in \Hom(T,G)} (A_G)_{f(e_u),f(e_v)} \right) - (n-m)\hom(T,G) \\       
 &= \sum_{e \in E_0} \hom(H_e,G) - (n-m)\hom(T,G).
    \end{align*}
    
    When \(G= K_{d+1}\) or when $H$ is bipartite and $G = K_{d,d}$, all of the relevant \((A_G)_{f(e_u),f(e_v)}=1\) and thus there is equality.   

    Finally \eqref{part:4} follows by combining \eqref{part:2} and \eqref{part:3}. 
\end{proof}

\subsection{The recursive process}

Here we prove \cref{thm:bound} by induction on the size of $H$. The only non-tree graph on at most $3$ nodes is $H = C_3$. But clearly \[\inj(C_3,G) = \hom(C_3,G) = \sum_{\lambda\in \sigma(G)} \lambda^3\] so the result is true. In the bipartite case the smallest non-tree graph is $C_4$.  
that satisfies all the other requirements as well. The only non-tree bipartite graph on at most 4 nodes is $H= C_4$. Here 
\[\inj(C_4,G) = \hom(C_4,G) - 2 \hom(P_3,G)+\hom(P_2,G) = \sum_{\lambda \in \sigma(G)} \lambda^4 -2d^2 + d,\] so, again the result is true. 

We will describe the inductive step in the non-bipartite case first. Suppose $|V(H)| = n$ and that we already proved the existence of the required polynomials for all smaller graphs. Let $H' \leq H$ be a uni-cyclic subgraph of $H$ that is also not bipartite. Then we can combine \cref{lem:hom_ineq} \eqref{part:1} with \cref{prop:hom to inj} twice to get
\begin{multline*}
     \inj(H,G) \leq \inj(H',G) = \hom(H',G) - \sum_{P} \inj(H'/P,G) \\= \hom(H',G) - \sum_P \hom(H'/P,G) + \sum_P \sum_Q \inj(H'/P/Q)
\end{multline*}
We can apply \cref{lem:hom_ineq} \eqref{part:2} to write $\hom(H',G) = \sum_{\lambda \in \sigma(G)} \lambda^k d^{n-k}$ where $k$ is the length of the unique cycle in $H'$ and hence odd. Next, we can apply \cref{lem:hom_ineq} \eqref{part:4} to each $-\hom(H/P,G)$ term in the first sum to get 
\[ -\sum_P \hom(H'/P,G) \leq \sum_{\lambda \in \sigma(G)} q_{H'/P}(\lambda,d)\]
where $q_{H'/P}(\lambda,d) = (|V(H'/P)|-|E(H'/P)|)d^{|V(H'/P|-1} - \sum_j c_j(H'/P) \lambda^j$ is simply the right-hand side of the inequality in \cref{lem:hom_ineq} \eqref{part:4} applied to $H'/P$. Note that these $q_{H'/P}$ polynomials have degree strictly less than $n = |V(H)|$. Finally, we can apply the inductive hypothesis to each $H'/P/Q$ to get a polynomial $p_{H'/P/Q}$ of degree $|V(H'/P/Q)|<n $ that upper-bounds $\inj(H'/P/Q,G)$ with equality for $G = K_{d+1}$. 

Combining these, we get that 
\[ p_H(\lambda,d) = \lambda^k d^{n-k} + \sum_P q_{H'/P}(\lambda,d) + \sum_P \sum_Q p_{H'/P/Q}(\lambda,d)\] is a degree $n$ polynomial such that for some odd $k$ the polynomial $p_H - \lambda^k d^{n-k}$ has degree strictly less than $n$, and 
\[ \inj(H,G) \leq \sum_{\lambda \in \sigma(G)} p_H(\lambda,d)\] as required. Last, but not least, because all inequalities above are satisfied with equality for $G= K_{d+1}$, the same is true for this last one. This finishes the proof of the non-bipartite case.

\subsection{The bipartite case}

The main difference is that in this case we need to achieve that $p_H$ only contains even powers of $\lambda$. The difficulty lies in that even a bipartite $H$ can generate non-bipartite quotients $H/P$, which in turn could lead to odd $\lambda$ exponents. The key observation is that it is sufficient to find a bounding polynomial that works for all bipartite $G$. 

\begin{lem}
    Let $H$ be a bipartite graph. Suppose $p_H(\lambda,d)$ is a polynomial that only has even $\lambda$-exponents, and such that \cref{eq:inj_bound} holds for all bipartite $G\in \GG_d$. Then it holds for all $G \in \GG_d$.
\end{lem}
\begin{proof}
    Let $G \in \GG_d$ and define $\tilde{G}$ to be its ``bipartite double cover'': the vertex set $V(\tilde{G}) = \{ x, x' : x \in V(G)\}$ consists of two copies of each original node. The edges are lifted so that they cross between the two copies: $E(\tilde{G}) = \{ (xy'), (x'y) : (xy) \in E(G)\}$. Then $\tilde{G}$ is bipartite, and any injective homomorphism $f \in \Hom(H,G)$ ``lifts'' to two distinct injective homomorphisms $f_1, f_2 \in \Hom(H,\tilde{G})$. Let $A,B$ denote the two parts of $H$. Let $f_1(a) = f(a)$ for all $a \in A$ and $f_1(b) = f(b)'$ for all $b \in B$, and the exact opposite for $f_2$. There may be other injective homomorphisms in $\Hom(H,\tilde{G})$, but at least we get $\inj(H,\tilde{G}) \geq 2 \inj(H,G)$.  At the same time, it is a simple and well-known fact that $\sigma(\tilde{G}) = \{ \lambda, -\lambda : \lambda \in \sigma(G)\}$ as a multi set.  Hence 
    \begin{multline*} \inj(H,G) \leq \frac{1}{2} \hom(H,\tilde{G}) \leq \frac{1}{2} \sum_{\lambda \in \sigma(\tilde{G})} p_H(\lambda, d) \\= \sum_{\lambda \in \sigma(G)} \frac{p_H(\lambda,d)+p_H(-\lambda,d)}{2} = \sum_{\lambda \in \sigma(G)} p_H(\lambda,d)\end{multline*}
    where the last step depended on $p_H$ only having even $\lambda$-exponent terms, hence $p_H(\lambda,d) = p_H(-\lambda,d)$.
\end{proof}

So now we may assume that $G$ is bipartite. That means that for a non-bipartite quotient $H'/P$ in the recursive process the term $\inj(H'/P,G)$ is automatically 0, so we can simply drop it. In other words (using the same notation as in the previous section) we obtain the following:
\begin{multline*}
     \inj(H,G) \leq \inj(H',G) = \hom(H',G) - \sum_{\stackrel{P : H'/P}{\textnormal{bipartite}}} \inj(H'/P,G) \\= \hom(H',G) - \sum_{\stackrel{P : H'/P}{\textnormal{bipartite}}} \hom(H'/P,G) + \sum_{\stackrel{P : H'/P}{\textnormal{bipartite}}} \sum_{\stackrel{Q : H'/P/Q}{\textnormal{bipartite}}} \inj(H'/P/Q)
\end{multline*}
From here the recursive process is the same as in the non-bipartite case. However, since all graphs are bipartite, all cycles are even and we only get even exponents of $\lambda$ as desired. The part about equality in the case of $G = K_{d,d}$ also follows in the same way using the relevant parts of \cref{lem:hom_ineq}.

\section{The Optimization Lemmas}\label{sec:opt}

In this subsection we prove \cref{lem:even_opt} and \cref{lem:odd_opt}. At the end, we show \cref{prop:vertexwise}.

\begin{proof}[Proof of \cref{lem:even_opt}]
Recall that \(X\) denotes a random variable with values in \([-d,d]\) and with 0 mean and second moment \(d\). Let \(Y = (X/d)^2\). Then \(0\leq Y \leq 1\) and \(\E Y = 1/d\). Furthermore, it is easy to see that for any such \(Y\) there is a unique \(-d \leq X \leq d\) with \(\E X =0, \E X^2 = d\) such that \(Y= (X/d)^2\). 
Let us expand the given degree \(n\) polynomial \(p\) as   
\[p(x,d) = x^{2m} d^{n-2m} + \sum_{j \geq 1} \sum_{k=0}^{(n-j)/2} c_{j,k} x^{2k} d^{n-j-2k} \] and define \(q(y,d)\) as 
\[q(y,d) = \frac{1}{d^n} \cdot p(d\sqrt{y},d) = y^m + \sum_{j \geq 1} \sum_{k=0}^{(n-j)/2} c_{j,k} y^k d^{-j}. \]
With this notation, we can write 
\[ \E p(X,d) = d^n \E q(Y,d),\] 
so it suffices to show that for large enough \(d\) we have \(\E q(Y,d) \leq \E q(Y_{K{d,d}},d)\). Here, naturally, \(Y_{K_{d,d}}= (X_{K_{d,d}}/d)^2\) is the random variable whose value is 1 with probability \(1/d\) and 0 otherwise. 

Let us temporarily fix \(d\), and abbreviate \(q(y) := q(y,d)\). Consider the linear function 
\begin{align*}
    L(y) = q(0)+y\cdot (q(1)-q(0))
\end{align*}
that intersects the graph of \(q(y)\) at \(y=0\) and \(y=1\). Suppose that 
\begin{equation}\label{eq:Lq_bound}
    L(y) \geq q(y)\ \mbox{ on \([0,1]\) and }\ L(y) > g(y)\ \mbox{ on \((0,1)\).}
\end{equation}
    Then, clearly, 
\begin{multline*} \E q(Y) \leq \E L(Y) = q(0) + ((q(1)-q(0))\E Y =\\= q(0) + ((q(1)-q(0))\E Y_{K_{d,d}} = \E L(Y_{K_{d,d}}) = \E q(Y_{K,{d,d}}), \end{multline*}
since \(\E Y = \E Y_{K_{d,d}} = 1/d\), and since \(Y_{K_{d,d}}\) takes its values in \(\{0,1\}\) and \(L = q\) at these points. And equality holds if and only if \(Y\) is concentrated on \(\{0,1\}\). But since \(\E Y = 1/d\), this only happens if \(Y = Y_{K{d,d}}\) and thus \(X = X_{K_{d,d}}\).

It remains to show that \cref{eq:Lq_bound} holds for \(d\) large enough. This will essentially follow from the structure of \(q\), that all but the first term contains \(d^j\) for some \(j \geq 1\), so all these terms become small when \(d\) is large. More precisely, for any \(0< y <1\)

\begin{alignat*}{2}
&L(y)& = q(0) + y\cdot (q(1)-q(0)) &> q(y) \\
&\Leftrightarrow& \frac{q(1) - q(0)}{1} &> \frac{q(y) - q(0)}{y} = y^{m-1} + \sum_{j \geq 1}\sum_{k=1}^{(n-j)/2} c_{j,k} y^{k-1} d^{-j}\\
&\Leftrightarrow& 1-y^{m-1} &> \sum_{j \geq 1} \sum_{k=1}^{(n-j)/2} c_{j,k} (y^{k-1}-1) d^{-j}\\
&\Leftrightarrow& 1+y+y^2+\dots+y^{m-2} &> \sum_{j \geq 1} \sum_{k=1}^{(n-j)/2} -c_{j,k} (1+y+y^2+\dots+y^{k-2}) d^{-j} 
\end{alignat*}
The left hand side is at least 1, while the right hand side is a fixed length finite sum where each term contains at least a \(1/d\). Thus, clearly, setting \(d\) large enough will achieve that the supremum of the right hand side is less than 1. 
\end{proof}

\begin{proof}[Proof of \cref{lem:odd_opt}] 
There are only minor differences compared to the previous proof. We set \(Y = X/d\) so that \(-1 \leq Y \leq 1, \E Y =0\), and \(\E Y^2 = 1/d\). It is convenient to expand \(p\) and define \(q\) as follows:
\[p(x,d) = x^{2m+1} d^{n-2m-1} + \sum_{j \geq 1} \sum_{k=0}^{(n-j} c_{j,k} x^k d^{n-j-k}, \] 
\[q(y,d) = \frac{1}{d^n} \cdot p(d\cdot y,d) = y^{2m+1} + \sum_{j \geq 1} \sum_{k=0}^{n-j} c_{j,k} y^k d^{-j}. \]
Then, as before, \(\E p(X,d) = d^n \E q(Y,d)\), and again it suffices to show that \(\E q(Y,d) \leq \E q(Y_{K_{d+1}})\) for \(d\) large enough. Here \(Y_{K_{d+1}} = X_{K_{d+1}}/d\) is the random variable taking \(1\) with probability \(1/d\) and \(-1\) otherwise.

The main difference now is that since the first and second moments are fixed, we need to consider majorizing parabolas instead of lines. We will look for a quadratic polynomial \(L(y)\) that is tangent to \(q\) at \(y_0=-1/d\) and intersects \(q\) at 1. Hence, using the \(q(y) = q(y,d)\) abbreviation, we let 
\[ L(y) = q(y_0) + (y-y_0) q'(y_0) + (y-y_0)^2\cdot \frac{q(1) - q(y_0) - (1-y_0)q'(y_0)}{(1-y_0)^2}. \]
Later we will show that \(L(y) > q(y)\) for all \(-1 \leq y < 1, y\neq -1\) when \(d\) is large enough. Once we know this, we can argue similarly to the last proof. Since \(\E L(Y)\) depends only on the first two moments of \(Y\), and these are equal for \(Y\) and \(Y_{K_{d+1}}\), we find
\[ \E q(Y) \leq \E L(Y)  = \E L(Y_{K_{d+1}}) = \E q(Y_{K_{d+1}}),\]
and equality holds if and only if \(Y\) is concentrated on \(\{-1/d,1\}\), hence \(Y \sim Y_{K_{d+2}}\).

It remains to show that \(L(y) > q(y)\) on \((-1, 1)\setminus\{-1/d\}\). The reason is, again, similar to the previous proof: it is true for the leading term \(y^{2m+1}\) and the rest are small due to \(d\) being large. The calculation, however, is considerably more complicated. \(L(y) > q(y)\) is equivalent to 
\[ \frac{q(1) - q(y_0) - (1-y_0)q'(y_0)}{(1-y_0)^2} > \frac{q(y) - q(y_0) - (y-y_0)q'(y_0)}{(y-y_0)^2}.\]
Let 
\[ \hat{q}(y) := \frac{q(y) - q(y_0) - (y-y_0)q'(y_0)}{(y-y_0)^2},\] then our goal is to show that \( \hat{q}(y) < \hat{q}(1)\) for all \(-1 \leq y<1\), which is equivalent to 
\[ \frac{\hat{q}(1)-\hat{q}(y)}{1-y}  > 0.\]
\begin{claim} A simple, but lengthy computation shows that for \(y\neq y_0\)
\[ \frac{y^k - y_0^k - (y-y_0) k y_0^{k-1}}{(y-y_0)^2} = \sum_{a=0}^{k-2} (a+1)y_0^a y^{k-2-a}.\] 
\end{claim}
Plugging this into the expansion of \(q\) we get that 
\begin{multline*}
\hat{q}(y) = \frac{q(y) - q(y_0) - (y-y_0)q'(y_0)}{(y-y_0)^2} =\\ = \sum_{a=0}^{2m-1} (a+1) y_0^a y^{2m-1-a} + \sum_{j \geq 1} \sum_{k=0}^{n-j}\sum_{a=0}^{k-2} c_{j,k} (a+1) y_0^a y^{k-2-a} d^{-j}
\end{multline*}  is also a polynomial in \(y\) of total degree \(n-2\). Recall that \(y_0 = -1/d\), hence there is a single monomial in \(\hat{q}\) that does not contain at least one factor of \(1/d\), and that term comes from the first sum when picking \(a=0\).  Thus, \(\hat{q}(y)\) is of the form
\[ \hat{q}(y) = y^{2m-1} + \sum_{j\geq 1}\sum_{k=0}^{n-2-j} \hat{c}_{k,j} y^k d^{-j},\] so 
\[ \frac{\hat{q}(1)-\hat{q}(y)}{1-y} = 1+y+\dots+y^{2m-2} + \sum_{j\geq 1}\sum_{k=0}^{n-2-j} \hat{c}_{k,j}(1+y+\dots+y^{k-1})d^{-j} \]
The \(1+y+\dots+y^{2m}\) portion has a positive infimum on \([-1,1]\) while the second part is a fixed polynomial where each term contains at least one \(1/d\) factor. So its supremum on \([-1,1]\) is \(O(1/d)\). Thus, for large enough \(d\) this expression is positive on \([-1,1]\), and that finishes the proof.
\end{proof}

\begin{claim}\label{cl:all_d}
When \(p=\lambda^k\), the above proofs can be considerably shortened and improved: Descartes' rule of signs applied to the polynomial \(q(y) - L(y)\) shows that there will not be further intersections between \(q\) and \(L\) other than the two that are there by design. As such, \(L > q\) holds on the claimed interval regardless of the value of \(d\).
\end{claim}

It is easy to show that \(K_{d,d}\) maximizes the number of closed walks of even length vertex-wise, in a similar fashion to the proof of \cref{prop:Petersen}. We show that the same holds true for \(K_{d+1}\) and odd length walks.

\begin{proof}[Proof of \cref{prop:vertexwise}]
Let \(k\) be odd. The number of length-\(k\) closed walks from \(v\) to \(v\) in a graph \(G\) is equal to \((A_G)^k_{v,v}\), which can be written as 
\[ A^k_{v,v} = \sum_{i=1}^n \lambda_i^k \phi_i(v)^2,\] where \(\phi_i\) is the unit eigenvector of \(A_G\) corresponding to the eigenvalue \(\lambda_i\). Clearly \(\sum_i \phi(v)^2 = 1\), thus we can let \(X\) be the random variable whose value is \(\lambda_i\) with probability \(\phi_i(v)^2\). Then \(-d\leq X \leq d\) and
\begin{align*} E X &= \sum_i \lambda_i \phi_i(v)^2 = (A_G)_{v,v} = 0  &&\mbox{and}& \E X^2 &= \sum_i \lambda_i^2 \phi_i(v)^2 = (A_G)^2_{v,v} = d.
\end{align*}
Thus, Claim~\ref{cl:all_d} implies that \[(A_G)^k_{v,v} = \E X^k \leq \E X_{K_{d+1}}^k = (A_{K_{d+1}})^k_{w,w}\] for all \(d\geq 2\),  with equality if and only if the component of \(v\) is isomorphic to \(K_{d+1}\).
\end{proof}

\appendix

\section{Maximizing 5-cycles} \label{sec:appendix}
\subsection{Example for 5-cycles}
In this section, we explicitly show \cref{thm:odd_cycles_main} for \(H=C_5\), show the computation from \cref{prop:inj to hom}, and express this in terms of the spectrum of \(G\). Moreover, we compute how large \(d\) has to be for the result to hold.

Notice that, up to isomorphism, the only quotients of \(C_5\) with no self loops are \(C_5\), \(K_3 + e\), and \(K_3\), where \(K_3+e\) is a triangle with an antenna. These have one, five, and five partitions, respectively.

\begin{prop}\label{5-cycles}
For \(d\geq7\) and \(n=c(d+1)\), the \(d\)-regular graph with the highest \(C_5\)-density is copies of \(K_{d+1}\). For \(d=3\), then the optimal graph is copies of the Petersen graph.
\end{prop}

\begin{proof}
    By \cref{prop:inj to hom}, we calculate
    \begin{align*}
    \inj(C_5) = \hom(C_5) - 5\hom(K_3 + e) + 5\hom(K_3),
    \end{align*}

    Then, since \(G\) is \(d\)-regular, we have \(\hom(K_3 + e) = d\hom(K_3)\). Hence, by \cref{lem:odd_opt},
    \begin{align*}
    \max  \sum_{i=1}^n \lambda_i^5 + (5-5d)\lambda_i^3
    \end{align*}
    is solved by the spectrum of copies of \(K_{d+1}\), for \(d\) sufficiently large. For \(d=3\), a standard calculation shows that the solution is the spectrum of the Petersen graph, which is spectrally unique. This gives us a stronger version of \cref{prop:Petersen}.
    
    By inspection, for $H=C_5$, the threshold from the optimization \cref{lem:odd_opt} is \(d=7\) as the quadratic from \cref{lem:odd_opt} is above \(x^5+(5-5d)x^3\) for \(|x|\leq d\) for \(d\geq 7\).
\end{proof}

For \(d=4,5,6\) the solution to the optimization problem does not correspond to the spectrum of a graph. Moreover, for \(d=4,5\) we have found examples of graphs with a higher \(5\)-cycle density than $K_{d+1}$, and for $d=6$ we have examples that match \(K_7\). For \(d=4\), the octahedral graph, the circulants \(C_k(2,3)\) for \(k=7,9,12,13\) as well as \(C_7(1,2), \ K_3 \square K_3 \), and the graphs in \cref{fig:d=4 examples} all achieve the highest 5-cycle density we have found after an extensive search. For \(d=5\), the best we have found is the complement of \(K_3\sqcup C_5\). Lastly, for \(d=6\), we found that \(K_7, K_{3,3,3}\), and \(K_8\) minus a perfect matching all have the same 5-cycle density. 

We believe that for general $H$ and smaller $d$, it is easier to find counterexamples because the range where \(q(x)<p(x)\) is larger, thus making it easier to find graphs with these `optimal' eigenvalues.
\begin{figure}
    \centering
    \begin{tikzpicture}[scale=0.6]
      \node[circle, fill=black, inner sep=1.5pt] (0) at (-1, 0) {};
      \node[circle, fill=black, inner sep=1.5pt] (1) at (1, 0) {};
      \node[circle, fill=black, inner sep=1.5pt] (2) at (0, -1.73) {};
      \node[circle, fill=black, inner sep=1.5pt] (3) at (-2.73, -1) {};
      \node[circle, fill=black, inner sep=1.5pt] (4) at (-1.73, -2.73) {};
      \node[circle, fill=black, inner sep=1.5pt] (5) at (1.73, -2.73) {};
      \node[circle, fill=black, inner sep=1.5pt] (6) at (2.73, -1) {};
      \node[circle, fill=black, inner sep=1.5pt] (7) at (1, 2) {};
      \node[circle, fill=black, inner sep=1.5pt] (8) at (-1, 2) {};
      
      \draw (0) -- (2);
      \draw (2) -- (1);
      \draw (1) -- (0);
      \draw (0) -- (8);
      \draw (8) -- (7);
      \draw (7) -- (1);
      \draw (1) -- (6);
      \draw (6) -- (5);
      \draw (5) -- (2);
      \draw (2) -- (4);
      \draw (4) -- (3);
      \draw (3) -- (0);
      \draw (8) -- (3);
      \draw (4) -- (5);
      \draw (6) -- (7);
      \draw (5) -- (8);
      \draw (7) -- (4);
    \end{tikzpicture}
    \hspace{1cm}
    \begin{tikzpicture}[rotate=90, scale=0.5]
      \node[circle, fill=black, inner sep=1.5pt] (0) at (0, 3) {};
      \node[circle, fill=black, inner sep=1.5pt] (1) at (-2, 2) {};
      \node[circle, fill=black, inner sep=1.5pt] (2) at (2, 2) {};
      \node[circle, fill=black, inner sep=1.5pt] (3) at (-3, 0) {};
      \node[circle, fill=black, inner sep=1.5pt] (4) at (-3, -2) {};
      \node[circle, fill=black, inner sep=1.5pt] (5) at (-2, -4) {};
      \node[circle, fill=black, inner sep=1.5pt] (6) at (0, -5) {};
      \node[circle, fill=black, inner sep=1.5pt] (7) at (2, -4) {};
      \node[circle, fill=black, inner sep=1.5pt] (8) at (3, -2) {};
      \node[circle, fill=black, inner sep=1.5pt] (9) at (3, 0) {};
      
      \draw (0) -- (1);
      \draw (0) -- (3);
      \draw (0) -- (6);
      \draw (0) -- (8);
      \draw (1) -- (2);
      \draw (1) -- (7);
      \draw (1) -- (8);
      \draw (3) -- (9);
      \draw (3) -- (5);
      \draw (3) -- (7);
      \draw (4) -- (6);
      \draw (4) -- (7);
      \draw (4) -- (8);
      \draw (4) -- (9);
      \draw (5) -- (8);
      \draw (5) -- (9);
      \draw (5) -- (2);
      \draw (6) -- (9);
      \draw (6) -- (2);
      \draw (7) -- (2);
    \end{tikzpicture}
    \caption{4-regular graphs with high \(C_5\) density}
    \label{fig:d=4 examples}
\end{figure}

\begin{proof}[Proof of \cref{prop:Petersen}]
    We show that the maximum number of 5-cycles through an edge \(e=(u,v)\) of a 3-regular graph \(G\) is 5 and that all such extremal graphs have a very specific structure around $e$.

    All 5-cycles through \(e\) are of the form
    \begin{align*}
        y \rightarrow x \rightarrow v \stackrel{e}{\rightarrow} u \rightarrow z \rightarrow y.
    \end{align*}
Since \(G\) is 3-regular, there are at most four pairs of nodes that can play the role of \(x,z\). Denote the two other neighbors of $v$ by \(x_i : i=1,2\) and the two other neighbors of $u$ by \(z_j : j=1,2\).  The number of 5-cycles through any given pair \((x_i,z_j)\) is 0 if they coincide, and otherwise it is the number of common neighbors between them. 

There are three possibilities of how the \(x_i\) and the \( z_j\) can coincide. We consider each case separately.

\begin{enumerate}
\item If both \(x_1=z_1, x_2=z_2\), then only the \((x_1,z_2)\) and \((x_2,z_1\) pairs can contribute to 5-cycles, and only 1 each because they only have one neighbor that is not \(u,v\). Thus in this case there are only two 5-cycles through \(e\). (See \cref{fig:3}). 
\item Next, without loss of generality, suppose \(x_1=z_1\) but \(x_2 \neq z_2\). Then, there are at most three distinct pairs that can support a 5-cycle: \((x_1, z_2), (x_2, x_1)\), and \((x_2, z_2)\). But the node \(x_1\) has only one neighbor different from $u,v$, hence the \((x_1, z_2)\) and \(( z_1,x_2 )\) pairs can only have one 5-cycle each. Finally, the \((x_2, z_2)\) pair can have at most two 5-cycles (see  \cref{fig:3} for the maximal case), which still only gives four 5-cycles through \(e\).

\item Now, suppose \(x_i\neq z_j\) for \(i=1,2\). In order for $e$ to have at least five 5-cycles, one pair must be involved in two of them. Let this pair be \(x_1,z_1\). Then there are two nodes $y_1,y_2$ that are both common neighbors of this pair. Now both $y_1,y_2$ can only be adjacent to one of $x_2,z_2$ each, since they have degree 3. So the $(x_1,z_2)$ and $(x_2,z_1)$ pairs can only support at most one 5-cycle each. If the $(x_2,z_2)$ pair also supports two 5-cycles, then there must exist nodes $y_3,y_4$ that are common neighbors of $x_2,z_2$, hence they are distinct from $y_1,y_2$ and thus the $(x_1,z_2)$ and $(x_2,z_1)$ pairs support zero 5-cycles.  

To summarize, in order to achieve at least five 5-cycles, we need at least one pair \((x_1,z_2)\) to support two of those, but if we have another pair also supporting two then that limits us to at most four 5-cycles in total. So the only way to get five is to have one supported on each of the other pairs. But this forces the exact structure of 
\cref{fig:figure 1}, which has at least 9 vertices and a cut-edge.  
\end{enumerate}

We have shown that the maximum number of 5-cycles through \(e\) is 5 and any maximal case has the exact local structure of \cref{fig:figure 1}, including a cut edge with no 5-cycles through it. It follows from the rigidity of that structure that no other edge can, in fact, support five 5-cycles, and the cut edge supports 0. Thus the local 5-cycle density is strictly less than 4.

Thus copies of the Petersen graph indeed have the largest possible density of 5-cycles, which is 4.
\end{proof}
\begin{figure}[H]
    \centering
    \begin{tikzpicture}
      \node[circle, fill=black, inner sep=1.5pt, label=below:$u$] (0) at (-1, -2) {};
      \node[circle, fill=black, inner sep=1.5pt, label=below:$v$] (1) at (1, -2) {};
      \node[circle, fill=black, inner sep=1.5pt,label=below:$x_1$] (2) at (0, -0.73) {};
      \node[circle, fill=black, inner sep=1.5pt, label=right:$x_2$] (3) at (0, 0) {};
      \node[circle, fill=black, inner sep=1.5pt] (4) at (-1, -0.5) {};
      
      \node[circle, draw=white, fill=white, inner sep=1.5pt] (5) at (-1.5, 0.25) {};
      
      \draw (0) -- (1);
      \draw (1) -- (2);
      \draw (2) -- (0);
      \draw (0) -- (3);
      \draw (3) -- (1);
      \draw (2) -- (4);
      \draw (3) -- (4);
      \draw (5) -- (4);
    \end{tikzpicture}
    \hspace{2cm}
    \begin{tikzpicture}
      \node[circle, fill=black, inner sep=1.5pt] (0) at (90:0.5) {};
      \node[circle, fill=black, inner sep=1.5pt, label=left:$z_2$] (1) at (162:1) {};
      \node[circle, fill=black, inner sep=1.5pt, label=below:$u$] (2) at (234:1) {};
      \node[circle, fill=black, inner sep=1.5pt, label=below:$v$] (3) at (306:1) {};
      \node[circle, fill=black, inner sep=1.5pt, label=right:$x_2$] (4) at (18:1) {};
      \node[circle, fill=black, inner sep=1.5pt, label=right:$x_1$] (5) at (0,0){};
      \node[circle, fill=black, inner sep=1.5pt] (6) at (0, 1){};
      \node[circle, draw = white, fill=white, inner sep=1.5pt] (7) at (-0.75, 1.25){};

      \draw (0) -- (1);
      \draw (1) -- (2);
      \draw (2) -- (3);
      \draw (3) -- (4);
      \draw (4) -- (0);

      \draw (0) -- (5);
      \draw (2) -- (5);
      \draw (3) -- (5);

      \draw (0) -- (6);
      \draw (1) -- (6);
      \draw (4) -- (6);
      \draw (6) -- (7);
    \end{tikzpicture}
    \caption{Maximal cases for \(x_1=z_1,~x_2=z_2\) and \(x_1=z_1,~x_2 \neq z_2\).}
    \label{fig:3}
\end{figure}

\printbibliography
\end{document}